\documentclass[a4paper,11pt]{article}
\usepackage[english]{babel}
\usepackage[utf8]{inputenc}
\usepackage[left=25mm,right=25mm,top=25mm,bottom=25mm]{geometry}

\usepackage{fixme}
\fxsetup{
    status=draft,
    author=,
    layout=footnote, 
    theme=color
}
\definecolor{fxnote}{rgb}{0.8000,0.0000,0.0000}

\usepackage[leqno]{amsmath}
\usepackage{amsthm,amssymb,amsfonts,mathabx}
\usepackage{mathrsfs}
\usepackage{nicefrac}

\usepackage{float}

\newcommand\sL{\mathcal{L}}

\newcommand{\N}{\mathbb{N}}
\newcommand{\Z}{\mathbb{Z}}
\renewcommand{\S}{\mathcal{S}}
\newcommand{\M}{\mathcal{M}}
\newcommand{\C}{\mathcal{C}}
\newcommand{\G}{\mathcal{G}}

\makeatletter
\newcommand{\leqnomode}{\tagsleft@true}
\newcommand{\reqnomode}{\tagsleft@false}
\makeatother

\usepackage{tikz}

\newtheorem{theorem}{Theorem}[section]
\newtheorem{lemma}[theorem]{Lemma}
\newtheorem{observation}[theorem]{Observation}

\newtheorem{definition}[theorem]{Definition}

\newtheorem{remark}[theorem]{Remark}
\newtheorem{problem}[theorem]{Open Problem}

\usepackage{hyperref}

\title{Subsystems of Open Induction\footnote{Research supported by the Austrian Science Fund (FWF) Projekt No.\ P35787}}
\author{Stefan Hetzl\footnote{Corresponding author: \url{stefan.hetzl@tuwien.ac.at}}, Johannes Weiser\\
Institute of Discrete Mathematics and Geometry\\
TU Wien}

\begin{document}

\maketitle
\begin{abstract}
We study subsystems of open induction obtained from restricting induction to
 atoms, literals, clauses, and dual clauses.
We obtain an almost complete picture of the relationships between these systems
 in the language of arithmetic and its sublanguages
 in terms of inclusion and strict inclusion.
The motivation for studying these systems is that they are strongly
 connected to methods of automated inductive theorem proving.
As corollaries we obtain practically relevant independence results for
 such methods.
\end{abstract}

\section{Introduction}

Theories of arithmetic play a central role in mathematical logic.
They are intimately connected to various topics ranging from the incompleteness
 theorems~\cite{Smullyan92Goedel,Hajek93Metamathematics} to computational complexity~\cite{Buss86Bounded,Krajicek19Proof}.
Weak theories of arithmetic, such as, for example, open induction, are closely
 related to certain classes of rings~\cite{Shepherdson64Nonstandard,Shepherdson67Rule,Kaye91Models}.

Weak theories of arithmetic are also of central importance for understanding
 the logical strength of methods for automated inductive theorem proving.
The aim of this area of computer science is to develop algorithms that find formal proofs by
 induction automatically.
Even though such algorithms would have a plethora of applications, e.g., in software verification
 and formal mathematics, the problem of how to design efficient such algorithms is, despite
 decades of research, still largely unsolved.
The connection to weak theories of arithmetic has been established through a number
 of recent results, for example:
equational theory exploration methods, such as HipSpec~\cite{Claessen13Automating}, are bounded in strength
 by induction on atoms~\cite[Chapter 8]{Vierling24Limits}.
The strength of the extension of the Vampire theorem prover presented in~\cite{Hajdu20Induction}
 is bounded by
 induction on literals~\cite{Hetzl23Induction}.
The strength of multi-clause induction in Vampire~\cite{Hajdu21Induction} is bounded by a
 certain form of induction
 on clauses~\cite[Chapter 4]{Vierling24Limits}.
The strength of clause set cycles, a theoretical model of methods such as the $n$-clause
 calculus~\cite{Kersani13Combining}, is bounded by a certain restriction of the
 parameter-free $\exists_1$ induction rule~\cite{Hetzl22Unprovability}.

In mathematical logic, open induction has been studied since, at least, \cite{Shoenfield58Open}
 and~\cite{Shepherdson64Nonstandard,Shepherdson65Nonstandard,Shepherdson67Rule}.
In this paper, we carry out a systematic study of subsystems of open induction along the
 lines of~\cite{Shepherdson65Nonstandard}.
We consider sublanguages of the language of arithmetic and, for each of these sublanguages, the
 subsystems of open induction obtained from
 restricting induction formulas to atoms, literals, clauses, and dual clauses respectively.
This choice of restrictions of induction is explained by the importance of these classes
 of formulas in automated deduction.
Aside from the open Problem~\ref{problem:kovalyov}, we give a complete characterisation of the relationships between these systems in terms
 of inclusion and strict inclusion.
These results provide a solid logical basis for the study of methods of automated
 inductive theorem proving for the language of arithmetic.
Through the connection to automated inductive theorem proving, the early results
 of~\cite{Shoenfield58Open,Shepherdson64Nonstandard,Shepherdson65Nonstandard,Shepherdson67Rule}
 attain fresh relevance and new applications.

As a by-product, our results provide a number of {\em practically relevant independence results}.
By this we mean the following: obviously there are many true statements which are not provable in certain
 subsystems of open induction.
The easiest way of exhibiting concrete such examples is to rely on results from the literature concerncing
 stronger systems, such as, for example, Peano arithmetic (PA).
Thus, for example, the consistency of PA~\cite{Smorynski77Incompleteness} and the Paris-Harrington
 theorem~\cite{Paris77Mathematical} are independent from open induction
 or any of its subsystems considered in this paper.
However, from the point of view of automated inductive theorem proving, these statements are
 not convincing, since they would be considered outside the scope of automation for purely practical reasons anyway.
By a practically relevant independence result we mean one whose sentence is of a content and a
 level of complexity that one would, at least in principle, consider to be within the scope of automation.
A good example is Shepherdson's result that the irrationality of $\sqrt{2}$
 is independent of open induction~\cite{Shepherdson64Nonstandard}.

This paper is based on the second author's master's thesis~\cite{Weiser25Subsystems}.
Another type of subsystems of open induction, which are defined by restricting the polynomial
 degree of the induction formulas, has been studied in~\cite{Boughattas91Arithmetique,Mohsenipour07Note}.

\section{Preliminaries}


In this paper we will work in first-order logic with equality.
We work with the language $\sL = \{0,s, p, +, \cdot\}$ or subsets of it.
$0$ is 0-ary, $s$ and $p$ are unary and $+$ and $\cdot$ are binary.
For the sake of readability, we will often write $sx$ and $px$ instead of $s(x)$ and $p(x)$
 respectively.
If it is clear from the context, we might also drop the $\cdot$ and write $xy$ instead
 of $x \cdot y$.
Equality ($=$) is a binary predicate with the usual axiomatization.
Let $\Sigma$ be a language.
A set of sentences $T$ over $\Sigma$ is called a $\Sigma$-theory.
If $\Sigma$ is clear from the context, we call $T$ a theory.
For any theory $T$, we write $\mathrm{Th}(T)$ for the deductive closure of $T$, i.e.,
 $\mathrm{Th}(T) = \{\varphi \mid \varphi \text{ is a formula and } T \vdash \varphi\}$.
For a model $\M$ we write $a\in \M$ to denote that $a$ is an element of the domain of $\M$.
Let $\M$ be any model over the language $\sL' \subseteq \sL$ with $0,s \in \sL'$. The standard part $\S_\M$ of the model $\M$ is defined as $\{(s^n0)^\M \mid n \geq 0\}$.
An element $a \in \M$ is a standard element if $A \in \S_\M$ and a non-standard element otherwise.

Often, it makes sense to consider the graph induced by the constructor $s$.
\begin{definition}
    Let $\{0,s\} \subseteq \sL' \subseteq \sL$ be a language and let $\M$ be an $\sL'$ structure.
    The directed graph $\G = (V,E)$ induced by $\M$ is given by the domain $V$ of $\M$ and the edge relation $E$ defined by $a \mathbin{E} b \Leftrightarrow b = s^\M a$. The weakly connected components of the undirected reduct of $\G$ are called comparison classes.
\end{definition}

Our base axioms are the following $\sL$ formulas:
\begin{flalign}
    \tag*{$\mathbf{A1}$} &s(x) \neq 0 && \label{ax:arithmetics_a1} \\
    \tag*{$\mathbf{A2}$} &p(0) = 0 && \label{ax:arithmetics_a2} \\
    \tag*{$\mathbf{A3}$} &p(s(x)) = x && \label{ax:arithmetics_a3} \\
    \tag*{$\mathbf{A3a}$} &sx = sy \rightarrow x = y && \label{ax:arithmetics_a3a} \\
    \tag*{$\mathbf{A4}$} &x + 0 = x && \label{ax:arithmetics_a4} \\
    \tag*{$\mathbf{A5}$} &x + s(y) = s(x+y) && \label{ax:arithmetics_a5} \\
    \tag*{$\mathbf{A6}$} &x\cdot 0 = 0 && \label{ax:arithmetics_a6} \\
    \tag*{$\mathbf{A7}$} &x \cdot s(y) = x\cdot y + x && \label{ax:arithmetics_a7}
\end{flalign}

We have the following scheme for induction:
\begin{flalign}
    & \varphi(0, \overline{z}) \land \forall x : \varphi(x, \overline{z}) \rightarrow \varphi(sx, \overline{z})\label{ax:arithmetics_lhs}\tag*{$\mathbf{LHS}(\varphi(x, \overline{z}))$} && \\
    & \hyperref[ax:arithmetics_lhs]{\mathbf{LHS}(\varphi(x, \overline{z}))} \rightarrow \forall x : \varphi(x, \overline{z}) \label{ax:arithmetics_induction}\tag*{$\mathbf{I}(\varphi(x, \overline{z}))$} &&
\end{flalign}

In the scheme above $\overline{z}$ are the parameters of the formula $\varphi$. Again, for the purpose of legibility, we might sometimes not mention parameters explicitly as all formulas may contain them if not stated otherwise.

Note that, in the list of the axioms above, we have the axioms $\hyperref[ax:arithmetics_a3]{A3}$ and $\hyperref[ax:arithmetics_a3a]{A3a}$. The first one does not state directly that $s$ is injective, but it follows trivially from it. Thus, whenever we have the axiom $\hyperref[ax:arithmetics_a3]{A3}$, we can also use $\hyperref[ax:arithmetics_a3a]{A3a}$ freely.

Having established the general context, we can consider various theories over $\sL$ and its subsets in the following sections. We start by taking the empty theory and gradually expand it. In addition, the language will be enriched step by step. For each fixed language we are interested in how certain subsets of open induction relate to each other.

\begin{definition}
    Let $\Sigma$ be any language.
    An atom is a formula that contains no quantifiers or logical connectives.
    A literal is an atom or a negated atom.
    A clause is a disjunction of literals.
    A dual clause is a conjunction of literals.
    An open formula is a quantifier-free formula.
\end{definition}
\begin{definition} Let $\Sigma$ be any language. We define:
    \begin{align*}
        \mathrm{IAtom} &= \{\hyperref[ax:arithmetics_induction]{\mathbf{I}(A)} \mid A \text{ is an atom}\} \\
        \mathrm{ILiteral} &= \{\hyperref[ax:arithmetics_induction]{\mathbf{I}(L)} \mid L \text{ is a literal}\}\\
        \mathrm{IClause} &= \{\hyperref[ax:arithmetics_induction]{\mathbf{I}(C)} \mid C \text{ is a clause}\}\\
        \mathrm{IDClause} &= \{\hyperref[ax:arithmetics_induction]{\mathbf{I}(D)} \mid D \text{ is a dual clause}\} \\
        \mathrm{IOpen} &= \{\hyperref[ax:arithmetics_induction]{\mathbf{I}(F)} \mid F \text{ is an open formula}\}
    \end{align*}
\end{definition}
By construction, these sets of formulas are partially ordered by the subset relation.
Moreover, we can consider the binary relation $\preceq$ on the set of all theories over a
 given language defined by $\Gamma \preceq \Delta \Leftrightarrow \mathrm{Th}(\Gamma) \subseteq \mathrm{Th}(\Delta)$. The relations $\approx$ and $\precneq$ are inherited as usual: $\Gamma \approx \Delta\Leftrightarrow \Gamma \preceq \Delta \text{ and } \Delta \preceq \Gamma$ and $\Gamma \precneq \Delta\Leftrightarrow \Gamma \preceq \Delta \text{ and } \neg (\Gamma \approx \Delta)$.
Thus, the subsets of open induction are partially ordered by the relation $\preceq$ with the following Hasse diagram for any theory $B$:
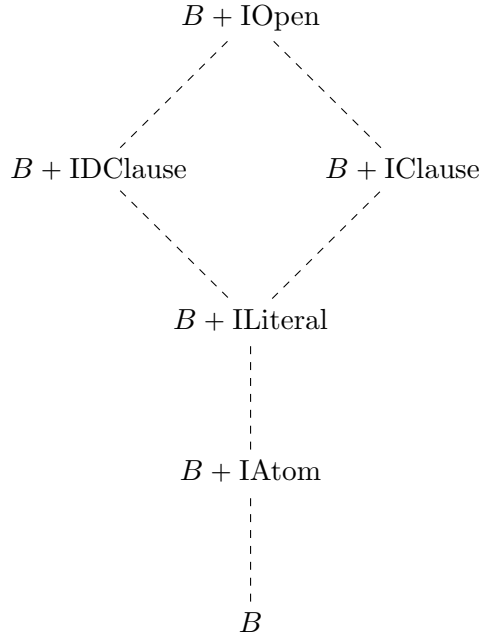
\begin{center}
    \begin{tikzpicture}[scale=1]
      \node (IOpen) at (0,2) {$B + \mathrm{IOpen}$};
      \node (IClause) at (2,0) {$B + \mathrm{IClause}$};
      \node (IDClause) at (-2,0) {$B + \mathrm{IDClause}$};
      \node (ILiteral) at (0,-2) {$B + \mathrm{ILiteral}$};
      \node (IAtom) at (0,-4) {$B + \mathrm{IAtom}$};
      \node (Base) at (0,-6) {$B$};
      \draw[dashed] (Base) -- (IAtom) -- (ILiteral) -- (IClause) -- (IOpen) -- (IDClause) -- (ILiteral);
    \end{tikzpicture}
\end{center}
In diagrams such as the above, dashed lines represent $\preceq$ and solid lines
 represent $\precneq$.
Clearly, we cannot draw any solid lines in the above general diagram, but we will need the
 distinction for the specific settings in later sections.

Before we proceed, we define an auxiliary axiom and prove a general result about it, which we will often refer to:
\begin{flalign}
    & x = 0 \lor\exists y : x = sy \label{ax:arithmetics_sur}\tag*{$\mathbf{SUR}$} &&
\end{flalign}
\begin{theorem}
\label{theorem:general_empty_theory_iliteral_sur}
    Let $\sL$ be any language extending $\{0,s\}$. Then $\mathrm{ILiteral} \vdash \hyperref[ax:arithmetics_sur]{\mathrm{SUR}}$.
\end{theorem}
\begin{proof}
    Take any model $\M$ of $\mathrm{ILiteral}$. Assume that there is some element $0\neq b \in \M$ s.t.\ $b$ does not lie in the image of $s^\M$. Consider the literal $L(x,z) \equiv x \neq z$. Then $\M \vDash \hyperref[ax:arithmetics_lhs]{\mathbf{LHS}(L(x,b))}$, but clearly, $\M \nvDash \forall x : L(x,b)$. Thus, $\M \nvDash \mathrm{ILiteral}$, which is a contradiction.
\end{proof}

\section{Two basic models}

\subsection{\texorpdfstring{The model $\N_\infty$}{The model N infinity}}
In this subsection, we will define a model that is in some sense the \textit{least non-standard model} of the natural numbers. We will prove some properties of this model and make heavy use of this model in the rest of this paper.

In this section we consider the language $\sL = \{0,s,p, +, \cdot\}$ and the theory $T = \{\hyperref[ax:arithmetics_a1]{A1}, \dots, \hyperref[ax:arithmetics_a7]{A7}\}$.
\begin{definition}
    The model $\N_\infty$ is constructed in the following way: The domain is given by $\N \cup \{\infty\}$. The symbols are interpreted in the following way:
    \begin{itemize}
        \item $0$ is interpreted as $0$
        \item All the symbols $s,p,+,\cdot$ are interpreted canonically in the standard part of the model
        \item $s(\infty) = p(\infty) = \infty$
        \item For any $x$ in the domain, $\infty + x = x+\infty = \infty$
        \item $\infty \cdot 0 = 0\cdot\infty = 0$ and for any $x \neq 0$, $x \cdot\infty = \infty \cdot x = \infty$
    \end{itemize}
\end{definition}

\begin{lemma}
\label{lemma:arithmetics_ninf_model}
    $\N_\infty \vDash T$
\end{lemma}
\begin{proof}
    This holds trivially by construction.
\end{proof}

\begin{lemma}\label{lemma:term_polynomial}
Let $t(x, \overline{z})$ be an $\sL$ term.
Let $\overline{b}$ be a tuple of natural numbers.
Then there is some $N \in \N$ and a polynomial $q \in \Z[x]$ s.t.\
 $\N_\infty \vDash t(n,\overline{b}) = q(n)$ for all $n \geq N$.
Moreover, if $x$ occurs in $t$, then $q$ is not constant.
\end{lemma}
\begin{proof}
The polynomial and the lower bound are constructed inductively, keeping track of the assertion
 that the polynomial is not constant at each step.
If $t \equiv x$, $t \equiv z_i$ or $t \equiv 0$, $N_t = 0$ and it is clear how to choose $q_t$.
If $t = pr$ for some term $r$, then set $q_t$ to be $q_r - 1$ and $N_t$ to be $N_r + 1$.
If $t = sr$, then $q_t = q_r+1$ and $N_t = N_r$.
If $t = r+u$, then $q_t = q_r+q_u$ and $N_t = \max{\{N_r, N_u\}}$.
Analogously for $\cdot$.
\end{proof}

\begin{lemma}
\label{lemma:arithmetics_ninf_equations}
Let $A(x,z_1,\ldots,z_k) \equiv t_1(x,\overline{z}) = t_2(x,\overline{z})$ be any atom.
Let $m_1,\ldots,m_k \in \N_\infty$.
If $\N_\infty \vDash A(n,\overline{m})$ for all $n \in \N$, then
 $\N_\infty \vDash \forall x : A(x,\overline{m})$.
\end{lemma}
\begin{proof}
1.\  If $x$ does not occur in $A$, then
 $\N_\infty \vDash A(0,\overline{z}) \leftrightarrow A(\infty,\overline{z})$ and hence
 $\N_\infty \vDash \forall x : A(x,\overline{m})$.

2.\ If $x$ occurs on exactly one side of $A$, say $t_1$, then
 $a \mapsto t_2^{\N_\infty}(a,\overline{m})$ is constant on $\N_\infty$.
We distinguish two cases.

2a.\ If $a \mapsto t_1^{\N_\infty}(a,\overline{m})$ is constant on $\N_\infty$, then
 $\N_\infty \vDash \forall x \forall x' (A(x,\overline{m}) \leftrightarrow A(x',\overline{m}))$
 and thus $\N_\infty \vDash \forall x : A(x,\overline{m})$.

2b.\ If $a \mapsto t_1^{\N_\infty}(a,\overline{m})$ is not contant on $\N_\infty$, then $x$
 occurs in $t_1$ and there is no $i\in \{ 1,\ldots,k\}$ with $m_i = \infty$ s.t.\ $z_i$
 occurs in $t_1$.
So, by Lemma~\ref{lemma:term_polynomial}, there is an $N\in \N$ and a non-constant polynomial
 $q \in \Z[x]$ s.t.\ $\N_\infty t_1(n,\overline{m}) = q(n)$ for all $n \geq N$.
In particular, there is some $n\in \N$
 s.t.\ $\N_\infty \nvDash t_1(n,\overline{m}) = t_2(\overline{m})$.

3.\ If $x$ occurs on both sides, then, by construction of the model,
 $\N_\infty \vDash t_1(\infty,\overline{z}) = \infty = t_2(\infty,\overline{z})$, i.e.,
 $\N_\infty \vDash \forall \overline{z} \forall x: A(x,\overline{z})$.
\end{proof}
%
%
%
%

\begin{lemma}
\label{lemma:arithmetics_ninf_iatom}
    For any reduct $\N_\infty^*$ of $\N_\infty$, whose language contains $0$ and $s$, we have $\N_\infty^* \vDash \mathrm{IAtom}$.
\end{lemma}
\begin{proof}
    W.l.o.g.\ assume that $\N_\infty^* = \N_\infty$. The claim follows directly from Lemma~\ref{lemma:arithmetics_ninf_equations} as $\N_\infty \vDash A(\underline{n})$ for any $n \in \N$ if $\N_\infty \vDash A(0) \land A(x)\rightarrow A(sx)$.
\end{proof}
\begin{lemma}
\label{lemma:arithmetics_ninf_iliteral}
    For any reduct $\N_\infty^*$ of $\N_\infty$, whose language contains $0$ and $s$,
    we have $\N_\infty^* \nvDash \mathrm{ILiteral}$.
\end{lemma}
\begin{proof}
    Consider the literal $L(x,z) \equiv x \neq z$. Clearly, $\N_\infty^* \vDash L(0,\infty) \land L(x,\infty) \rightarrow L(sx, \infty)$, but $\N_\infty^*\nvDash \forall x : L(x, \infty)$. Note that this works regardless of the concrete reduct, as $L$ uses no symbols of the language.
\end{proof}
\begin{theorem}
\label{theorem:arithmetics_ninf_main}
    Let $T' \subseteq T$ be some theory over the language $\sL' \subseteq \sL$. Then $T' + \mathrm{IAtom} \nvdash \mathrm{ILiteral}$
\end{theorem}
\begin{proof}
    By Lemma~\ref{lemma:arithmetics_ninf_model} the appropriate reduct $\N_\infty^*$ of $\N_\infty$ is a model of $T'$. By Lemma~\ref{lemma:arithmetics_ninf_iatom} $\N_\infty^*$ is a model of $\mathrm{IAtom}$, but by Lemma~\ref{lemma:arithmetics_ninf_iliteral},  $\N_\infty^*$ is not a model of $\mathrm{ILiteral}$.
\end{proof}

\subsection{\texorpdfstring{The model $\N_{a,b}$}{The model Nab}}
In this section, we stick with the language $\sL = \{0,s,p, +, \cdot\}$ and the theory $T = \{\hyperref[ax:arithmetics_a1]{A1}, \dots, \hyperref[ax:arithmetics_a7]{A7}\}$. We will define another very common non-standard model of $T$. We will use this model to separate atomic induction from the base theory in some cases.

\begin{definition}
\label{definition:arithmetics_nab}
    The model $\N_{a,b}$ is constructed in the following way: The domain is given by $\N \cup \{a,b\}$. The symbols are interpreted as follows: 
    \begin{itemize}
        \item $0$ is interpreted as $0$
        \item For every standard element $n$, $s^{\N_{a,b}}n = n+1$, $s^{\N_{a,b}}a = a$ and $s^{\N_{a,b}}b = b$
        \item For every standard element $n \neq 0$, $p^{\N_{a,b}}n = n-1$, $p^{\N_{a,b}}0 = 0$, $p^{\N_{a,b}}a = a$, $p^{\N_{a,b}}b = b$ 
        \item $+$ and $\cdot$ are interpreted according to the tables below
    \end{itemize} 
    \begin{table}[H]
    \begin{center}
        \begin{tabular}{l|llll|ll}
            + & 0 & 1 & 2 & \dots & a & b \\ \hline
            0 & 0 & 1 & 2 & \dots & b & a  \\
            1 & 1 & 2 & 3 & \dots & b & a  \\
            2 & 2 & 3 & 4 & \dots & b & a  \\
            $\vdots$ & $\vdots$ & $\vdots$ & $\vdots$ & $\ddots$ & $\vdots$ & $\vdots$   \\ \hline
            a & a & a & a & \dots & a & a   \\
            b & b & b & b & \dots & b & b 
        \end{tabular}
        \hspace{2cm}
        \begin{tabular}{l|llll|ll}
            $\cdot$ & 0 & 1 & 2 & \dots & a & b \\ \hline
            0 & 0 & 0 & 0 & \dots & b & a  \\
            1 & 0 & 1 & 2 & \dots & b & a  \\
            2 & 0 & 2 & 4 & \dots & b & a  \\
            $\vdots$ & $\vdots$ & $\vdots$ & $\vdots$ & $\ddots$ & $\vdots$ & $\vdots$   \\ \hline
            a & 0 & b & b & \dots & a & a   \\
            b & 0 & a & a & \dots & b & b
        \end{tabular}
    \end{center}
    \end{table}
\end{definition}
\begin{lemma}
    $\N_{a,b} \vDash T$
\end{lemma}
\begin{proof}
Straightforward verification.
\end{proof}
\begin{lemma}
\label{lemma:arithmetics_nab_plus_non_comm}
$\N_{a,b} \nvDash x+y = y+x$
\end{lemma}
\begin{proof}
$\N_{a,b}\vDash a+0 = a \neq b = 0+a$.
\end{proof}

\section{Zero and Successor only}

\label{section:arithmetics_t0}

For this section we fix the language $\sL = \{0,s\}$.
\begin{definition}
    We define two very basic theories:
    \begin{itemize}
        \item $T_0 = \emptyset$
        \item $T_1 = \{\hyperref[ax:arithmetics_a1]{A1}\}$
    \end{itemize}
\end{definition}
Additionally, we define the following auxiliary axioms:
\begin{flalign}
    & s^n0 = s^{n+m}0 \rightarrow \forall x \bigvee_{k=0}^{n+m-1} x=s^k0 \text{ for any } n,m \in \N \text{ and } m \geq 1 \label{ax:arithmetics_bnm}\tag*{$\mathbf{B_{n,m}}$} && 
\end{flalign}
Informally these axioms can be formulated as the fact that if a model $\M$ of $\{\hyperref[ax:arithmetics_bnm]{B_{n,m}} \mid n,m \in \N, m \geq 1\}$ contains a \textit{lollipop} starting with $0$, then every element is a (transitive) successor of $0$.
In \cite{Shepherdson65Nonstandard} different axioms were used instead of $\hyperref[ax:arithmetics_bnm]{B_{n,m}}$: 
\begin{flalign}
    & s^n0 = s^{m+1}0 \rightarrow \forall x \bigvee_{k=0}^{m} x=s^k0 \text{ for any } n,m \in \N \text{ and } n \leq m \label{ax:arithmetics_bnm_alt}\tag*{$\mathbf{B_{n,m}'}$} && 
\end{flalign}
\begin{observation}
For each $n,m \in \N$ with $m \geq 1$, $\hyperref[ax:arithmetics_bnm]{B_{n,m}}$ is
 $\hyperref[ax:arithmetics_bnm_alt]{B_{n,n+m-1}'}$.
On the other hand, for each $n,m \in \N$ with $n \leq m$,
 $\hyperref[ax:arithmetics_bnm_alt]{B_{n,m}'}$ is $\hyperref[ax:arithmetics_bnm]{B_{n,m-n+1}}$.
\end{observation}

\begin{theorem}
\label{theorem:arithmetics_t0_shepherdson}
    In \cite{Shepherdson65Nonstandard}, it was shown that:
    \begin{itemize}
        \item $T_0 + \mathrm{IOpen} \approx T_0 + \{\hyperref[ax:arithmetics_bnm_alt]{B_{n,m}'} \mid n,m \in \N, n \leq m\}$
        \item $T_1 + \mathrm{IOpen} \approx T_1 + \{\hyperref[ax:arithmetics_bnm_alt]{B_{n,m}'} \mid n,m \in \N, n \leq m\}$
    \end{itemize}
    From the observation above it follows that:
    \begin{itemize}
        \item $T_0 + \mathrm{IOpen} \approx T_0 + \{\hyperref[ax:arithmetics_bnm]{B_{n,m}} \mid n,m \in \N, n \leq m\}$
        \item $T_1 + \mathrm{IOpen} \approx T_1 + \{\hyperref[ax:arithmetics_bnm]{B_{n,m}} \mid n,m \in \N, n \leq m\}$
    \end{itemize}
\end{theorem}
We can strengthen this by the following theorem:
\begin{theorem}
    For $i \in \{0,1\}$ we have:
    \begin{align*}
        T_i &\precneq T_i + \mathrm{IAtom} \\
        &\precneq T_i + \mathrm{ILiteral} \\
        &\precneq T_i + \mathrm{IDClause} \\
        &\approx T_i + \mathrm{IClause} \\
        &\approx T_i + \mathrm{IOpen} \approx T_i +  \{\hyperref[ax:arithmetics_bnm]{B_{n,m}} \mid n,m \in \N, n \leq m\}.
    \end{align*}
    This yields the following Hasse Diagram:
    \begin{center}
        \begin{tikzpicture}
            \node (Base) at (0,-2) {$T_i$};
            \node (IAtom) at (0,0) {$T_i + \mathrm{IAtom}$};
            \node (ILiteral) at (0,2) {$T_i + \mathrm{ILiteral}$};
            \node (IClause) at (0,4) {$T_i + \mathrm{IDClause} \approx T_i + \mathrm{IClause} \approx T_i + \mathrm{IOpen}$};
            \draw (Base) -- (IAtom) -- (ILiteral) -- (IClause);
        \end{tikzpicture}
    \end{center}
\end{theorem}
This theorem will follow directly from the following lemmas.
\begin{remark}
    Note that $T_0$ and $T_1$ are defined over the same language with $T_0 \subseteq T_1$. Thus, by showing $T_1 \nvdash \varphi$ for some formula $\varphi$, we obtain that $T_0 \nvdash \varphi$ for free and vice-versa for $T_0 \vdash \varphi$.
\end{remark}

\begin{lemma}
\label{lemma:arithmetics_t0_bnm}
$T_0 +\mathrm{IClause} \vdash \hyperref[ax:arithmetics_bnm]{B_{n,m}}$ for any $n,m \in \N$
 with $m \geq 1$.
\end{lemma}
\begin{proof}
Fix $n,m\in\N$ with $m \geq 1$.
We work in $T_0 + \mathrm{IClause}$.
Assume that $s^n 0 = s^{n+m} 0$.
We make an induction on the clause $C(x) \equiv \bigvee_{k=0}^{n+m-1} x = s^k 0$.
Clearly $C(0)$ is logically valid.
Now, assume $C(x)$.
Then there is a $k \leq n + m + 1$ s.t.\ $x = s^k0$.
If $k < n+m-1$, then $k+1 \leq n+m-1$ and $sx=s^{k+1}0$ and thus $C(sx)$.
If $k = n+m-1$, then $sx = s^{k+1}0=s^{n+m}0 =s^n0$ by the assumption.
Since $0 \leq n \leq n+m-1$ we have $C(sx)$.
\end{proof}
%
%
%

\begin{lemma}
    $T_1 \nvdash \mathrm{IAtom}$
\end{lemma}
\begin{proof}
    Consider the following model $\M$: The domain is given by $\{0,1,2,a,b,c\}$ and the symbols are interpreted in the following way:
    \begin{itemize}
        \item $0^\M = 0$
        \item $s^\M 0 = 1, s^\M 1 = 2, s^\M 2 = 1$
        \item $s^\M a = b, s^\M b = c, s^\M c = a$
    \end{itemize}
    Since $0$ has no predecessor in this model, all axioms of $T_1$ hold. Consider the atom $A(x) \equiv sx = s^3x$. Then $A$ holds for the elements $0,1,2$, but not for $a,b,c$. In particular, $\M \vDash \hyperref[ax:arithmetics_lhs]{\mathbf{LHS}(A)}$, but $\M \nvDash \forall x : A(x)$. Thus, $\M \nvDash \hyperref[ax:arithmetics_induction]{\mathbf{I}(A)}$.
\end{proof}
\begin{lemma}
    $T_1 + \mathrm{IAtom} \nvdash \mathrm{ILiteral}$
\end{lemma}
\begin{proof}
    This follows directly from Theorem~\ref{theorem:arithmetics_ninf_main}.
\end{proof}
Next, we want to show that $T_0 + \mathrm{IDClause} \vdash \mathrm{IClause}$.
For this, we need the following lemma:
\begin{lemma}
\label{lemma:arithmetics_t0_idclause_inf_pred}
    Let $\M \vDash T_0 + \mathrm{IDClause}$ and $a \in \M$ be non-standard.
    Then for every $n \in \N$, there is a sequence of distinct elements $b_0, \dots, b_n$ s.t.\ $s^\M b_i = b_{i+1}$ and $b_n= a$.
\end{lemma}
\begin{proof}
    Assume that there is a non-standard element $a$ and some $n$ s.t.\ such a sequence does not exist for $a$. Let $m$ be the biggest number s.t.\ such a sequence $b_0, \dots,b_m$ does exist for $a$ and consider the dual clause $D(x) \equiv \bigwedge_{i=0}^m x \neq b_i$. Since $a$ is non-standard, $\M \vDash D(0)$. If there is some element $c$ s.t.\ $D(c)$ holds and $D(sc)$ does not, then $sc = b_0$. Then, however, $c,b_0, \dots,b_m$ would be the sequence for $m+1$, which contradicts our assumption of $m$ being maximal. Thus, $\M \vDash D(x) \rightarrow D(sx)$. Induction on $D$ yields $\M \vDash \forall x : D(x)$, which contradicts $\M \vDash a = a$.
\end{proof}
\begin{lemma}
    $T_0 + \mathrm{IDClause} \vdash \mathrm{IOpen}$.
\end{lemma}
\begin{proof}
    From Theorem~\ref{theorem:general_empty_theory_iliteral_sur}, it follows that $T_0 + \mathrm{ILiteral} \vdash x = 0 \lor \exists y : x = sy$. Now let us analyze how models of $T_0 + \mathrm{IDClause}$ look like: Let $\M$ be any model with the domain $M$, $G=(M,E)$ be the graph induced by $s$ and $\C \subseteq 2^M$ the set of all connected components of $G$. We fix $C_0$ to be the component containing $0$ and distinguish two cases:

    1. Assume that there are no $n,m$ with $m \neq 0$ s.t.\ $s^n0 = s^{n+m}0$. Then, all of the $\hyperref[ax:arithmetics_bnm]{B_{n,m}}$ trivially hold. By Theorem $~\ref{theorem:arithmetics_t0_shepherdson}$, open induction holds in $\M$.

    2. Assume that there are $n,m \in \N$ with $m \geq 1$ s.t.\ $s^n0 = s^{n+m}0$. Consider the atom $A(x) \equiv s^nx = s^{n+m}x$. By assumption, $\M \vDash A(0)$. Take any element $b$ s.t.\ $\M \vDash A(b)$. Clearly, $\M \vDash s^n(sb) = s(s^nb) = s(s^{n+m}b) = s^{n+m}(sb)$ and thus, $\M \vDash A(sb)$. Induction on $A$ yields $\M \vDash \forall x : A(x)$. Now take any non-standard element $b$. By the observation above the set $\{a \in \M \mid \text{there is some } k\in \N \text{ s.t. } s^k(a) = b\}$ is finite. By Lemma~\ref{lemma:arithmetics_t0_idclause_inf_pred}, however, this set is infinite. Thus, there are no non-standard elements. In particular, $\M \vDash \hyperref[ax:arithmetics_bnm]{B_{n,m}}$. Again, by Theorem~\ref{theorem:arithmetics_t0_shepherdson}, open induction holds in $\M$.
\end{proof}
By Lemma~\ref{lemma:arithmetics_t0_bnm} and Theorem~\ref{theorem:arithmetics_t0_shepherdson}
 we also have $T_0 + \mathrm{IClause} \nvdash \mathrm{IOpen}$.
Let $i\in\{0,1\}$.
Since, trivially, $T_i + \mathrm{IOpen} \nvdash \mathrm{IClause}$ and
 $T_i + \mathrm{IOpen} \nvdash \mathrm{IDClause}$, we have
 $T_i + \mathrm{IClause} \approx T_i + \mathrm{IDClause} \approx T_i + \mathrm{IOpen}$.
It remains to show that $T_1 + \mathrm{ILiteral} \nvdash \mathrm{IClause}$.
\begin{lemma}\label{lemma:t0_ILiteral_nproves_IClause}
    $T_1 + \mathrm{ILiteral} \nvdash \mathrm{IClause}$.
\end{lemma}
\begin{proof}
    By Lemma~\ref{lemma:arithmetics_t0_bnm} it suffices to give a model of $T_1 + \mathrm{ILiteral}$ which does not satisfy $\hyperref[ax:arithmetics_bnm]{B_{n,m}}$ for some suitable $n,m$.
    Consider the following model $\M$: The domain is given by $\{0,1,2,a,b\}$. The symbols are interpreted in the following way:
    \begin{itemize}
        \item $0^\M = 0$
        \item $s^\M 0 = 1, s^\M 1 = 2, s^\M 2 = 1$
        \item $s^\M a = b, s^\M b = a$
    \end{itemize}
    Since $0$ is no successor, $T_1$ holds. Moreover, $\M \nvDash \hyperref[ax:arithmetics_bnm]{B_{1,2}}$. So, by Lemma~\ref{lemma:arithmetics_t0_bnm}, $\M \nvDash \mathrm{IClause}$.
    It remains to show that $\M \vDash \mathrm{ILiteral}$.

First, we analyze how terms are interpreted in this model: Any term $t$ has the form $s^n(y)$,
 where $y$ is either $x$, some parameter, or $0$.
Assume that $t \equiv s^nx$.
Then, there are three cases:
If $n = 0$, then $t(x) \equiv x$.
If $n>0$ and $n \equiv 0 \mod{2}$, then $t(0)^\M = t(2)^\M = 2, t(1)^{\M} = 1, t(a)^\M = a$
 and $t(b)^\M = b$.
If $n> 0$ and $n \equiv 1 \mod{2}$, then $t(0)^\M = t(2)^\M = 1, t(1)^\M = 2, t(a)^\M = b$
 and $t(b)^\M = a$.

Now consider an atom $A(x,z_1,z_2) \equiv t_1(x,z_1) = t_2(x,z_2)$.
We make a case distinction.

1.\ If $x$ occurs in neither $t_i$, then, for all $m_1, m_2 \in \M$:
 $\M \vDash \forall x (A(x,m_1,m_2) \leftrightarrow \top )$ or
 $\M \vDash \forall x (A(x,m_1,m_2) \leftrightarrow \bot )$, so induction over $A(x,m_1,m_2)$
 holds in $\M$.

2.\ If $x$ occurs in exactly one term, say $t_1$, then $A$ is $A(x,z_2) \equiv t_1(x) = t_2(z_2)$.
Fix $m_2 \in \M$.

2a.\ If $t_2^\M(m_2) \notin \mathrm{rng}(t_1^\M)$, then $\M \vDash \forall x (t_1(m) = t_2(m) \leftrightarrow \bot)$, so induction over $A(x,m_2)$ holds in $\M$.

2b.\ If $t_2^\M(m_2) \in \mathrm{rng}(t_1^\M)$, let $m_1 \in \M$ s.t.\ $t_1^\M(m_1) = t_2^\M(m_2)$.
Since $\M \vDash \forall x: t_1(x) \neq t_1(sx)$, we have $\M \nvDash t_1(m_1) = t_2(m_2) \rightarrow t_1(s m_1) = t_2(s m_2)$ and hence induction over $A(x,m_2)$ holds in $\M$.

3.\ If $x$ occurs in both terms, then $t_1(x) \equiv s^{k_1}x$, $t_2(x) \equiv s^{k_2}(x)$,
 and $A$ is $A(x) \equiv s^{k_1}(x) = s^{k_2}(x)$.
The function $t_i^\M$ is determined by the case distinction whether $k_i = 0$,
 $k_i \in 2\N + 1$, or $k_i \in 2\N + 2$.
Since each case yields a different interpretation of $t_i^\M(0)$, we have
 $\M \vDash \forall x (A(x) \leftrightarrow A(0))$, so $\M \vDash (\forall x A(x)) \leftrightarrow A(0)$ and thus induction over $A(x)$ holds in $\M$.
%
%
%
%

Now consider a negated atom $L(x,z_1,z_2) \equiv t_1(x,z_1) \neq t_2(x,z_2)$.
We make a case distinction.

1.\ If $x$ occurs in neither $t_i$, then, forall $m_1,m_2 \in \M$, induction over $A(x,m_1,m_2)$
 holds in $\M$ as above.

2.\ If $x$ occurs in exactly one term, say $t_1$, then $L$ is
 $L(x,z_2) \equiv s^{k_1} x \neq t_2(z_2)$.
Fix $m_2 \in \M$.

2a.\ If $t_2^\M(m_2) = 0$ and $n=0$, then $\M \nvDash L(0,m_2)$.

2b.\ If $t_2^\M(m_2) = 0$ and $n\neq 0$, then $\M \nvDash \forall x: L(x,m_2)$.

2c.\ If $t_2^\M(m_2) = c \neq 0$, then there are $d,e\in \M$ s.t.\ $\M \vDash sd = c$ and
 $\M \vDash s^{k_1} e = d$.
So $\M \vDash s^{k_1 + 1} e = c$ and, since $c \neq d$, $\M \nvDash L(e,m_2) \rightarrow L(se,m_2)$.

3.\ If $x$ occurs in both terms, then, as above, $\M \vDash \forall x (L(x) \leftrightarrow L(0))$
 and induction over $L(x)$ holds in $\M$.
%
%
%
%
%
%
\end{proof}

\begin{remark}\label{remark:t0_prindepresult}
The proof of Lemma~\ref{lemma:t0_ILiteral_nproves_IClause} shows that
 $T_1 + \mathrm{ILiteral} \nvdash B_{1,2}$ in order to separate
 $\mathrm{ILiteral}$ from $\mathrm{IClause}$.
This shows, as a corollary, that methods from automated inductive
 theorem proving that can be captured by $\mathrm{ILiteral}$, such
 as, for example, single clause induction in Vampire~\cite{Hajdu20Induction},
 cannot prove $B_{1,2}$ from $T_1$.
Similar results can be obtained for other $B_{n,m}$.
Since the sentences $B_{n,m}$ are of a complexity and of a content
 that one would regard as being in the scope of automated inductive
 theorem proving, this argument provides a {\em practically
 relevant independence result}.
\end{remark}

\section{Injective Successor}\label{section:arithmetics_t2}

For this section, we fix the language $\sL = \{0,s\}$, define the theory $T_2$
 and the additional axioms $B_n$.
\begin{definition}
    $T_2 = \{\hyperref[ax:arithmetics_a1]{A1}, \hyperref[ax:arithmetics_a3a]{A3a} \}$
\end{definition}
\begin{definition}
    We define the axiom
    \begin{flalign}
    & x \neq s^nx \text{ for any } n \geq 1. \label{ax:arithmetics_bn}\tag*{$\mathbf{B_n}$} &&
    \end{flalign}
\end{definition}
In \cite{Shepherdson65Nonstandard} the following was shown:
\begin{theorem}
\label{theorem:arithmetics_t2_shepherdson}
    $T_2 + \mathrm{IOpen}$ is equivalent to $T_2 + \{\hyperref[ax:arithmetics_bn]{B_n} \mid n \in \N, n \geq 1\}$.
\end{theorem}
We can refine this theorem as follows:
\begin{theorem}
\label{theorem:arithmetics_t2_main}
    \begin{align*}
        T_2 &\approx T_2 + \mathrm{IAtom} \\
        &\precneq T_2 + \mathrm{ILiteral} \\
        &\approx T_2 + \mathrm{IOpen} \\
        &\approx T_2 + \{\hyperref[ax:arithmetics_bn]{B_n} \mid n \geq 1\}
    \end{align*}
    This yields the following Hasse Diagram:
    \begin{center}
        \begin{tikzpicture}
            \node (Base) at (0,-2) {$T_2 \approx T_2 + \mathrm{IAtom}$};
            \node (ILiteral) at (0,0) {$T_2 + \mathrm{ILiteral} \approx T_2 + \mathrm{IOpen}$};
            \draw (Base) -- (ILiteral);
        \end{tikzpicture}
    \end{center}
\end{theorem}

\begin{lemma}\label{lem:T2_ILiteral_Bn}
    $T_2 + \mathrm{ILiteral} \vdash \hyperref[ax:arithmetics_bn]{B_n}$ for any $n \geq 1$.
\end{lemma}
\begin{proof}
    Fix some $n \geq 1$. We work in $T_2 + \mathrm{ILiteral}$. Consider the literal $L(x) \equiv x \neq s^nx$. By $\hyperref[ax:arithmetics_a1]{A1}$, $T_2 \vdash L(0)$. By $\hyperref[ax:arithmetics_a1]{A3a}$, $T_2 \vdash x\neq s^nx \rightarrow sx \neq s^{n+1}x$. Thus, $T_2 \vdash \hyperref[ax:arithmetics_lhs]{\mathbf{LHS}(L)}$ and $T_2 + \mathrm{ILiteral} \vdash \forall x : x \neq s^nx$.
\end{proof}

\begin{lemma}
    $T_2 \vdash \mathrm{IAtom}$
\end{lemma}
\begin{proof}
Take any model $\M$ of $T_2$ and any atom $A(x,z_1,z_2) \equiv t_1(x,z_1) = t_2(x,z_2)$.
We make a case distinction.

1.\ If neither $t_1$ nor $t_2$ contains $x$, then for all $m_1, m_2 \in \M$:
 $\M \vDash \forall x (A(x,m_1,m_2) \leftrightarrow \top)$ or
 $\M \vDash \forall x (A(x,m_1,m_2) \leftrightarrow \bot)$ and thus induction over
 $A(x,m_1,m_2)$ holds in $\M$.

2.\ If $x$ occurs in exactly one term, say $t_1$, then $A$ is $A(x,z_2) \equiv s^m(x) = t_2(z_2)$.
Fix $m_2 \in \M$ and assume $\M \vDash A(0,m_2)$, i.e., $\M \vDash s^{k_1}(0) = t_2(m_2)$.
From $\hyperref[ax:arithmetics_a3a]{A3a}$ and $\hyperref[ax:arithmetics_a1]{A1}$ it follows that
 $\M \vDash t_1(0) = s^{k_1} 0 \neq s^{k_1 + 1} 0 = t_1(s0)$.
So $\M \nvDash A(s0,m_2)$ and thus induction over $A(x,m_2)$ holds in $\M$.

3.\ If $x$ occurs in both terms, then $A$ is $A(x) \equiv s^{k_1}(x) = s^{k_2}(x)$.
Assume that $\M \vDash A(0)$.
Then, by $\hyperref[ax:arithmetics_a3a]{A3a}$ and $\hyperref[ax:arithmetics_a1]{A1}$, $k_1 = k_2$.
So $\M \vDash \forall x : A(x)$ and thus induction over $A(x)$ holds in $\M$.
\end{proof}
%
%
%

\begin{lemma}\label{lem:T2_IAtom_ILiteral}
    $T_2 + \mathrm{IAtom} \nvdash \mathrm{ILiteral}$
\end{lemma}
\begin{proof}
Let $\N_\infty^*$ be the reduct of $\N_\infty$ to $\sL = \{0,s\}$.
Then, by Lemma~\ref{lemma:arithmetics_ninf_model}, $\N_\infty^* \vDash T_2$.
Moreover, by Lemma~\ref{lemma:arithmetics_ninf_iatom}, $\N_\infty^* \vDash \mathrm{IAtom}$.
Furthermore, note that for all $n\geq 1$: $\N_\infty^* \nvDash B_n$.
Therefore, by Lemma~\ref{lem:T2_ILiteral_Bn}, $\N_\infty^* \nvDash \mathrm{ILiteral}$.
\end{proof}

\begin{remark}\label{remark:t2_prindepresult}
The proof of Lemma~\ref{lem:T2_IAtom_ILiteral} shows that for all $n\geq 1$:
 $T_2 + \mathrm{IAtom} \nvdash B_n$.
This shows, as a corollary, that methods from automated inductive theorem proving that are
 captured by atomic induction, such as equational theory exploration~\cite{Claessen13Automating},
 cannot prove $B_n$ from $T_2$.
\end{remark}

\section{Adding the Predecessor}

For this section, we fix the language $\sL = \{0,s,p\}$, define the theory $T_3$
 and the addtional axiom $B1$.
\begin{definition}
    $T_3 = \{\hyperref[ax:arithmetics_a1]{A1}, \hyperref[ax:arithmetics_a2]{A2}, \hyperref[ax:arithmetics_a3]{A3}\}$
\end{definition}
\begin{definition}
    We define the axiom:
    \begin{flalign}
        & x \neq 0 \rightarrow x = spx \label{ax:arithmetics_b1}\tag*{$\mathbf{B1}$} && 
    \end{flalign}
\end{definition}
In \cite{Shepherdson65Nonstandard} the following was proven:
\begin{theorem}
    $T_3 + \mathrm{IOpen}$ is equivalent to $T_3 + \{\hyperref[ax:arithmetics_bn]{B_n} \mid n \in \N, n\geq 1\} + \{\hyperref[ax:arithmetics_b1]{B1}\}.$
\end{theorem}
Again, we can strengthen this result:
\begin{theorem}
    \begin{align*}
        T_3 &\approx T_3 + \mathrm{IAtom} \\
        &\precneq T_3 + \mathrm{ILiteral} \\
        &\approx T_3 + \mathrm{IOpen} \\
        &\approx T_3 + \{\hyperref[ax:arithmetics_bn]{B_n} \mid n \in \N, n\geq 1\} + \{\hyperref[ax:arithmetics_b1]{B1}\}
    \end{align*}
    This yields the following Hasse diagram:
    \begin{center}
        \begin{tikzpicture}
            \node (Base) at (0,-2) {$T_3 \approx T_3 + \mathrm{IAtom}$};
            \node (ILiteral) at (0,0) {$T_3 + \mathrm{ILiteral} \approx T_3 + \mathrm{IOpen}$};
            \draw (Base) -- (ILiteral);
        \end{tikzpicture}
    \end{center}
\end{theorem}

\begin{lemma}
    \label{lemma:arithmetics_t3_iliteral_bn}
    $T_3 + \mathrm{ILiteral} \vdash \hyperref[ax:arithmetics_bn]{B_n}$ for all $n \geq 1$.
\end{lemma}
\begin{proof}
Fix any $n \geq 1$.
We work in $T_3 + \mathrm{ILiteral}$.
Consider the literal $L(x) \equiv x \neq s^nx$.
By \hyperref[ax:arithmetics_a1]{A1} we have that $L(0)$ holds.
By \hyperref[ax:arithmetics_a3]{A3} we have $sx = s^{n+1}x \rightarrow x = psx = ps^{n+1}x = ps^nx$
 and by contraposition $L(x) \rightarrow L(sx)$.
By induction on $L$, we obtain $\forall x : x \neq s^nx$.
\end{proof}

\begin{lemma}
\label{lemma:arithmetics_t3_iliteral_b1}
    $T_3 + \mathrm{ILiteral} \vdash \hyperref[ax:arithmetics_b1]{B1}$
\end{lemma}
\begin{proof}
    From Theorem~\ref{theorem:general_empty_theory_iliteral_sur} it follows that $T_3 + \mathrm{ILiteral} \vdash x = 0 \lor \exists y : x = sy$. Assume that $x \neq 0$. There is some $y$ s.t.\ $x = sy$. It follows that $x = sy \stackrel{\hyperref[ax:arithmetics_a3]{A3}}{=} s(psy) = (sp)sy = spx$
\end{proof}

\begin{lemma}
    $T_3 \vdash \mathrm{IAtom}$
\end{lemma}
\begin{proof}
Take any model $\M$ of $T_3$ and any atom $A(x,z_1,z_2) \equiv t_1(x,z_1) = t_2(x,z_2)$.
Note that, by applying $\hyperref[ax:arithmetics_a3]{A3}$, we can assume that both $t_i$ have the
 form $s^{k_i} p^{m_i} y_i$ where $y_i$ is either $0$, $x$, or $z_i$.
We make a case distinction.

1.\ If $x$ occurs in neither $t_i$, then induction over $A(x,z_1,z_2)$ holds trivially.

2.\ If $x$ occurs in exactly one term, say $t_1$, then $A$ is $A(x,z_2) \equiv s^{k_1} p^{m_1} x = s^{k_2} p^{m_2} z_2$.
We have $\M \vDash t_1(0) = s^{k_1} 0$ and $\M \vDash t_1(s^{m_1 + 1} 0) = s^{k_1+1} 0$.
So, by $\hyperref[ax:arithmetics_a3]{A3}$ and $\hyperref[ax:arithmetics_a1]{A1}$,
 $\M \vDash t_1(0) \neq t_1(s^{m_1 + 1} 0)$.
Fix $m_2 \in \M$.
Then $\exists i\in \{0,\ldots,m_1 \}$ s.t.\ $\M \vDash A(s^i 0, m_2)$ and
 $\M \nvDash A(s^{i+1} 0, m_2)$.
So induction over $A(x,m_2)$ holds in $\M$.

3.\ If $x$ occurs on both sides, then $A$ is $A(x) \equiv s^{k_1} p^{m_1} x = s^{k_2} p^{m_2} x$.
Then, by applying $\hyperref[ax:arithmetics_a3]{A3}$, we can assume that one of the $k_i$ is $0$.
Let w.l.o.g.\ $k_1 = 0$.
If $k_2 \neq 0$, then $\M \nvDash A(0)$ and thus
 $\M \nvDash \hyperref[ax:arithmetics_lhs]{\mathbf{LHS}(A)}$.
Assume that $k_2 = 0$.
Then $\M \vDash A(0)$.
Let w.l.o.g.\ $m_1 < m_2$.
Then $\M \vDash t_1(s^{m_1 + 1} 0) \neq t_2(s^{m_1 + 1} 0)$ and thus
 $\M \nvDash \hyperref[ax:arithmetics_lhs]{\mathbf{LHS}(A)}$.
In the remaining case $m_1 = m_2$, $A$ is an identity and thus $\M \vDash \hyperref[ax:arithmetics_induction]{\mathbf{I}(A)}$.
\end{proof}

%
%
%

\begin{lemma}
    $T_3 \nvdash \mathrm{ILiteral}$
\end{lemma}
\begin{proof}
    This follows directly from Theorem~\ref{theorem:arithmetics_ninf_main}.
\end{proof}

\section{Linear Arithmetic}\label{section:arithmetics_t4}

For this section, we fix the language $\sL = \{0,s,p,+\}$, define the theory $T_4$
 and the additional axioms $B2, B3, B4$.
\begin{definition}
    $T_4 = \{\hyperref[ax:arithmetics_a1]{A1},\hyperref[ax:arithmetics_a2]{A2}, \hyperref[ax:arithmetics_a3]{A3}, \hyperref[ax:arithmetics_a4]{A4}, \hyperref[ax:arithmetics_a5]{A5}\}$
\end{definition}
\begin{definition}
    We define the following axioms:
    \begin{flalign}
        & x+y = y+x \label{ax:arithmetics_b2}\tag*{$\mathbf{B2}$} && \\
        & (x+y)+z = x + (y+z) \label{ax:arithmetics_b3}\tag*{$\mathbf{B3}$} && \\
        & x+y = x+z \rightarrow y=z \label{ax:arithmetics_b4}\tag*{$\mathbf{B4}$} && 
    \end{flalign}
\end{definition}
In \cite{Shoenfield58Open} the following was shown:
\begin{theorem}
    $T_4 + \mathrm{IOpen}$ is equivalent to $T_4 + \{\hyperref[ax:arithmetics_b1]{B1},\hyperref[ax:arithmetics_b2]{B2},\hyperref[ax:arithmetics_b3]{B3},\hyperref[ax:arithmetics_b4]{B4}\}$
\end{theorem}
We can strengthen this result as follows:
\begin{theorem}
    \begin{align*}
        T_4 &\precneq T_4 + \mathrm{IAtom} \\
        &\precneq T_4 + \mathrm{ILiteral} \\
        &\approx T_4 + \mathrm{IOpen} \\
        &\approx T_ 4 +  \{\hyperref[ax:arithmetics_b1]{B1},\hyperref[ax:arithmetics_b2]{B2},\hyperref[ax:arithmetics_b3]{B3},\hyperref[ax:arithmetics_b4]{B4}\}
    \end{align*}
    This yields the Hasse diagram:
    \begin{center}
        \begin{tikzpicture}
            \node (Base) at (0,-2) {$T_4$};
            \node (IAtom) at (0,0) {$T_4 + \mathrm{IAtom}$};
            \node (ILiteral) at (0,2) {$T_4 + \mathrm{ILiteral} \approx T_4 + \mathrm{IOpen}$};
            \draw (Base) -- (IAtom) -- (ILiteral);
        \end{tikzpicture}
    \end{center}
\end{theorem}
\begin{lemma}
\label{lemma:arithmetics_t4_b1}
    $T_4 + \mathrm{ILiteral} \vdash \hyperref[ax:arithmetics_b1]{B1}$
\end{lemma}
\begin{proof}
    This follows directly from Lemma~\ref{lemma:arithmetics_t3_iliteral_b1} and the fact that $T_4$ is a superset of $T_3$.
\end{proof}

\begin{lemma}
\label{lemma:arithmetics_t4_b2}
    $T_4 + \mathrm{IAtom} \vdash \hyperref[ax:arithmetics_b2]{B2}$.
\end{lemma}
\begin{proof}
    First, apply induction on the atom $A_1(x) \equiv 0+x = x$. Then, make an induction on $x$ on the atom $A_2(x) \equiv sy + x = s(y+x)$. Lastly, apply induction on $x$ on the atom $A_3(x) \equiv x + y = y+x$.
\end{proof}

\begin{lemma}
\label{lemma:arithmetics_t4_b3}
    $T_4 + \mathrm{IAtom} \vdash \hyperref[ax:arithmetics_b3]{B3}$.
\end{lemma}
\begin{proof}
    Consider the atom $A(x,y,z) \equiv (x+y)+z = x+(y+z)$. Applying induction on $z$ on $A$ yields $\hyperref[ax:arithmetics_b3]{B3}$.
\end{proof}

\begin{lemma}
\label{lemma:arithemetics_t4_b4}
    $T_4 + \mathrm{ILiteral} \vdash \hyperref[ax:arithmetics_b4]{B4}$
\end{lemma}
\begin{proof}
    We work in $T_4 + \mathrm{ILiteral}$. Note that, by Lemma~\ref{lemma:arithmetics_t4_b2}, we can use commutativity.
    
    Assume that $y \neq z$ and consider the literal $L(x) \equiv x +y \neq x+z$. $L(0)$ holds because $0+y \stackrel{\hyperref[ax:arithmetics_b2]{B2}}{=} y+0 \stackrel{\hyperref[ax:arithmetics_a4]{A4}}{=} y \neq z \stackrel{\hyperref[ax:arithmetics_a4]{A4}}{=} z+0 \stackrel{\hyperref[ax:arithmetics_b2]{B2}}{=} 0+z$. Assume that $L(x)$ holds and $x+y \neq x+z$. Then $sx + y \stackrel{\hyperref[ax:arithmetics_b2]{B2}}{=} y+sx \stackrel{\hyperref[ax:arithmetics_a5]{A5}}{=} s(y+x) \stackrel{\hyperref[ax:arithmetics_b2]{B2}}{=} s(x+y) \stackrel{\hyperref[ax:arithmetics_a3a]{A3a}}{\neq} s(x+z) \stackrel{\hyperref[ax:arithmetics_b2]{B2}}{=} s(z+x) \stackrel{\hyperref[ax:arithmetics_a5]{A5}}{=} z+sx \stackrel{\hyperref[ax:arithmetics_b2]{B2}}{=} sx + z$.
    By induction over $L$, we obtain $\forall x : x +y \neq x+z$.
\end{proof}

\begin{lemma}
    $T_4 \nvdash \mathrm{IAtom}$
\end{lemma}
\begin{proof}
    By Lemma~\ref{lemma:arithmetics_t4_b2}, it suffices to give a model of $T_4$, in which $+$ is not commutative.
    The claim now follows directly from Lemma~\ref{lemma:arithmetics_nab_plus_non_comm} if we take the appropriate reduct of $\N_{a,b}$ from Definition~\ref{definition:arithmetics_nab}.
\end{proof}
\begin{lemma}
    $T_4 + \mathrm{IAtom} \nvdash \mathrm{ILiteral}$
\end{lemma}
\begin{proof}
    This follows directly from Theorem~\ref{theorem:arithmetics_ninf_main}.
\end{proof}

\section{Polynomials}\label{section:arithmetics_t5}

In this section, we will often consider (commutative) rings.
All these rings contain an identity, so we do not mention this explicitly every time.
The identity element will always be named $1$.
For this section, we fix the language $\sL = \{0,s,p,+,\cdot\}$, define the theory $T_5$,
 and the additional axioms $B5$, $B6$, $B7$, and $C'_d$.
\begin{definition}
    $T_5 = \{\hyperref[ax:arithmetics_a1]{A1},\ldots, \hyperref[ax:arithmetics_a7]{A7}\}$
\end{definition}
\begin{definition}
    We define the following axioms:
    \begin{flalign}
        & xy = yx \label{ax:arithmetics_b5}\tag*{$\mathbf{B5}$} && \\
        & x(yz) = (xy)z \label{ax:arithmetics_b6}\tag*{$\mathbf{B6}$} && \\
        & x(y+z) = xy+xz \rightarrow y=z \label{ax:arithmetics_b7}\tag*{$\mathbf{B7}$} && \\
        & dy = dz \rightarrow \bigvee_{i=0}^{d-1} (x+i)y = (x+i)z \text{ for any } d=2,3, \dots \label{ax:arithmetics_cd}\tag*{$\mathbf{C'_d}$} && 
    \end{flalign}
\end{definition}
The axioms $\hyperref[ax:arithmetics_cd]{C'_d}$ are not quite intuitive, but they can be
 understood as a very weak form of multiplicative cancellation.
Clearly, $\N \vDash (x \neq 0 \land xy = xz) \rightarrow y = z$, but
 $T_5 + \mathrm{IOpen} \nvdash (x \neq 0 \land xy = xz) \rightarrow y = z$.
This is a consequence of the following: In \cite{Shepherdson67Rule} it was shown that
 $T_5 + \{\hyperref[ax:arithmetics_b1]{B1},\ldots, \hyperref[ax:arithmetics_b7]{B7}\} + \{\hyperref[ax:arithmetics_cd]{C'_d} \mid d \in \N, d \geq 2\} \approx T_5 + IOpen$ and
 in \cite{Shepherdson65Nonstandard} it was shown that $T_5 +
 \{\hyperref[ax:arithmetics_b1]{B1} ,\ldots, \hyperref[ax:arithmetics_b7]{B7}\} + \{\hyperref[ax:arithmetics_cd]{C'_d} \mid d \in \N, d \geq 2\} \nvdash (x = 0 \land xy = xz) \rightarrow y=z$. However, as we will see below, $T_5 + \mathrm{IClause} \vdash \hyperref[ax:arithmetics_cd]{C'_d}$ for any $d \geq 2$.

The following result was postulated in \cite{Shepherdson65Nonstandard} and proven in \cite{Shepherdson67Rule}.
\begin{theorem}
\label{theorem:arithmetics_t5_shep}
    $T_5 + \mathrm{IOpen}$ is equivalent to $T_5 + \{\hyperref[ax:arithmetics_b1]{B1},\ldots, \hyperref[ax:arithmetics_b7]{B7}\} + \{C'_d \mid d \in \N, d \geq 2\}$.
\end{theorem}
We can strengthen this result:
\begin{theorem}
    \begin{align*}
        T_5 &\precneq T_5 + \mathrm{IAtom} \\
        &\precneq T_5 + \mathrm{ILiteral} \\
    \end{align*}
    This yields the following Hasse diagram:
    \begin{center}
        \begin{tikzpicture}
            \node (Base) at (0,-2) {$T_5$};
            \node (IAtom) at (0,0) {$T_5 + \mathrm{IAtom}$};
            \node (ILiteral) at (0,2) {$T_5 + \mathrm{ILiteral}$};
            \node (IDClause) at (0,4) {$T_5 + \mathrm{IDClause}$};
            \node (IClause) at (0,6) {$T_5 + \mathrm{IClause} \approx T_5 + \mathrm{IOpen}$};
            \draw (Base) -- (IAtom) -- (ILiteral);
            \draw[dashed] (ILiteral) -- (IDClause) -- (IClause);
        \end{tikzpicture}
    \end{center}
\end{theorem}
This theorem will be a consequence of the following lemmas in combination with Theorem~\ref{theorem:arithmetics_t5_shep}.
Note that the relationship between $T_5 + \mathrm{ILiteral}$, $T_5 + \mathrm{IDClause}$ and $T_5 + \mathrm{IClause}$ remains open. However, we will show that the unknown relationship between $T_5 + \mathrm{ILiteral}$ and $\mathrm{IClause}$ corresponds exactly to Problem 2 of~\cite{Kovalyov24Fragments}.
\begin{lemma}
\label{lemma:arithmetics_t5_iatom_b2}
    $T_5 + \mathrm{IAtom} \vdash \{\hyperref[ax:arithmetics_b2]{B2}, \hyperref[ax:arithmetics_b3]{B3}, \hyperref[ax:arithmetics_b5]{B5}, \hyperref[ax:arithmetics_b6]{B6}, \hyperref[ax:arithmetics_b7]{B7}\}$
\end{lemma}
\begin{proof}
    Since $T_5 \supseteq T_4$, it follows from Lemma~\ref{lemma:arithmetics_t4_b2} and Lemma~\ref{lemma:arithmetics_t4_b3} that $T_5 + \mathrm{IAtom} \vdash \hyperref[ax:arithmetics_b2]{B2}$ and $T_5 + \mathrm{IAtom} \vdash \hyperref[ax:arithmetics_b3]{B3}$. For the remaining proof, we work in $T_5 + \mathrm{IAtom}$.

    For $B5$ first apply induction on the atom $A_1(x) \equiv 0x = 0$. Then, apply induction on $x$ on the atom $A_2(x) \equiv (sy)x = yx + x$. Lastly, apply induction on $x$ on the atom $A_3(x) \equiv xy = yx$.

    For $B7$ apply induction on $x$ on the atom $A_4(x,y,z) \equiv x(y+z) = xy + xz$.

    For $B6$ apply induction on $x$ on the atom $A_5(x,y,z) \equiv (xy)z = x(yz)$.
\end{proof}

\begin{lemma}
    $T_5 + \mathrm{ILiteral} \vdash \{\hyperref[ax:arithmetics_b1]{B1}, \hyperref[ax:arithmetics_b4]{B4}\}$
\end{lemma}
\begin{proof}
    Since $T_5 \supseteq T_4$, this claim follows directly from Lemma~\ref{lemma:arithmetics_t4_b1} and Lemma~\ref{lemma:arithemetics_t4_b4}.
\end{proof}

\begin{lemma}
\label{lemma:arithmetics_t5_iclause_cd}
    $T_5 + \mathrm{IClause} \vdash \hyperref[ax:arithmetics_cd]{C'_d}$ for any $d \geq 2$.
\end{lemma}
\begin{proof}
    Fix any $d \geq 2$ and consider the clause $C(x) \equiv dy = dz \rightarrow \bigvee_{k=0}^{d-1}(s^kx)y = (s^kx)z$. We work in $T_5 + \mathrm{IClause}$. By \hyperref[ax:arithmetics_a6]{A6} and commutativity of $\cdot$, we have that $0y = 0 = 0z$ and in particular that $C(0)$ holds. Now assume that $C(x)$ holds. Take any $y,z$ s.t.\ $dy = dz$. If there is some $k > 0$ s.t.\ $(s^kx)y = (s^kx)z$, then $(s^{k-1}sx)y = (s^{k-1}sx)z$ and in particular $C(sx)$ holds. If $k = 0$, then $xy = xz$. Consider the term $(s^dx)y$. It holds that $(s^dx)y = (s^{d-1}x)y + y = \dots =(s^0x)y + dy = xy + dy = xz + dz = (s^dx)z$ and thus $C(sx)$. Induction on $C(x)$ yields the desired result.
\end{proof}
\begin{lemma}
    $T_5 \nvdash \mathrm{IAtom}$
\end{lemma}
\begin{proof}
    By Lemma~\ref{lemma:arithmetics_t5_iatom_b2} it suffices to give a model of $T_5$, where $+$ is not commutative. The claim now follows directly from Lemma~\ref{lemma:arithmetics_nab_plus_non_comm}.
\end{proof}

\begin{lemma}
    $T_5 + \mathrm{IAtom} \nvdash \mathrm{ILiteral}$
\end{lemma}
\begin{proof}
    This follows directly from Theorem~\ref{theorem:arithmetics_ninf_main}.
\end{proof}
The following definition and problem is taken from \cite[Problem 2]{Kovalyov24Fragments}:
\begin{definition}
    \begin{itemize}
        \item $\sL^\leq = \sL \cup \{\leq\}$
        \item $T_5^\leq = T_5 + \{\hyperref[ax:arithmetics_sur]{\mathrm{SUR}}, x \leq y \leftrightarrow \exists z : x+z = y\}$
        \item $\mathrm{IOpen}^\leq(=) = \{\hyperref[ax:arithmetics_induction]{\mathbf{I}(F)} \mid F \text{ is an open $\sL$-formula}\}$
        \item $\mathrm{I}^\leq(=, \neq) = \{\hyperref[ax:arithmetics_induction]{\mathbf{I}(L)} \mid L \text{ is an $\sL$-literal}\}$
    \end{itemize}
\end{definition}
\begin{problem}[Kovalyov]\label{problem:kovalyov}
    In the language $\sL^\leq$, is $T_5^\leq + \mathrm{I}^\leq(=,\neq) \approx T_5^\leq + \mathrm{IOpen}^\leq(=)$?
\end{problem}
\begin{lemma}
    Problem~\ref{problem:kovalyov} has a positive answer if and only if $T_5 + \mathrm{ILiteral} \approx T_5 + \mathrm{IClause}$
\end{lemma}
\begin{proof}
    First, for any $\sL^\leq$-sentence $\psi$, we define the translation $\psi^*$ by replacing each occurrence of $t_1 \leq t_2$ with $\exists y : t_1 + y = t_2$. The following properties hold:
    \begin{enumerate}
        \item For any $\sL^\leq$-sentence $\psi$, $\psi^* \equiv \psi$
        \item For any $\sL^\leq$-sentence $\psi$, $T_5^\leq \vdash \psi \leftrightarrow \psi^*$
        \item For any $\sL$-sentence $\psi$, if $T_5 + \mathrm{IClause} \vdash \psi$, then $T_5^\leq + \mathrm{IOpen}^\leq(=) \vdash \psi$
        \item For any $\sL$-sentence $\psi$, if $T_5 + \mathrm{ILiteral} \vdash \psi$, then $T_5^\leq + \mathrm{I}^\leq(=, \neq) \vdash \psi$
        \item For any $\sL^\leq$-sentence $\psi$, if $T_5^\leq + \mathrm{IOpen}^\leq(=) \vdash \psi$, then $T_5 + \mathrm{IClause} \vdash \psi^*$
        \item For any $\sL^\leq$-sentence $\psi$, if $T_5^\leq + \mathrm{I}^\leq(=, \neq) \vdash \psi$, then $T_5 + \mathrm{ILiteral} \vdash \psi^*$
    \end{enumerate}
    The properties 1, 2, 3 and 4 are trivial. Property 5 follows from the fact that $T_5 + \mathrm{IClause} \approx T_5 + \mathrm{IOpen}$ and from
    Theorem~\ref{theorem:general_empty_theory_iliteral_sur}.
    Property 6 follows analogously. In order to prove the lemma, we must consider
    two directions:

    $\Rightarrow$ : Assume that $T_5^\leq + \mathrm{I}^\leq(=,\neq)$ and $T_5^\leq + \mathrm{IOpen}^\leq(=)$ prove the same $\sL^\leq$-sentences. Now pick any $\sL$-sentence $\sigma$ with $T_5 + \mathrm{IClause} \vdash \sigma$. By 3, $T_5^\leq + \mathrm{IOpen}^\leq(=) \vdash \sigma$. By assumption, $T_5^\leq + \mathrm{I}^\leq(=, \neq) \vdash \sigma$. By 6 and 1, $T_5 + \mathrm{ILiteral} \vdash \sigma$.

    $\Leftarrow$ : Assume that $T_5 + \mathrm{ILiteral}$ and $T_5 + \mathrm{IClause}$ prove the same $\sL$-sentences. Now pick any $\sL^\leq$-sentence $\sigma$ with $T_5^\leq + \mathrm{IOpen}^\leq(=) \vdash \sigma$. By 5, $T_5 + \mathrm{IClause} \vdash \sigma^*$. By assumption, $T_5 + \mathrm{ILiteral} \vdash \sigma^*$. By 4 and 2, $T_5^\leq + \mathrm{I}^\leq(=, \neq) \vdash \sigma$. 
\end{proof}

%
%
\begin{remark}
    If Problem~\ref{problem:kovalyov} has a positive answer, then our diagram above is completed. Otherwise, there are still three possibilities to complete it since the position of $\mathrm{IDClause}$ is unclear. $\mathrm{IDClause}$ could either be contained in $\mathrm{ILiteral}$, in $\mathrm{IClause}$ or strictly between the two.
\end{remark}

\section{Conclusion}

Aside from Problem~\ref{problem:kovalyov}, we have obtained a complete characterisation, in terms of inclusion and strict inclusion, of
 the relationships of all subsystems of open induction considered in this paper.

Depending on the language and base theory, there is often some syntactic restriction of
 open induction which entails full open induction.
For example, with the base theory $T_2$ from Section~\ref{section:arithmetics_t2}, we have
 $T_2 + \mathrm{ILiteral} \vdash \mathrm{IOpen}$.
Moreover, the level of this restriction is not strictly increasing in the complexity of
 the language or theory in the following sense:
 Consider the theories $T_0, T_1$ over the same language as $T_2$ with $T_0 \subsetneq T_1 \subsetneq T_2$. For both $T_0$ and $T_1$, it holds that open induction is strictly stronger than induction over literals, but collapses to induction over clauses. Moreover, it is unclear, whether $T_5 + \mathrm{IOpen}$ might collapse to $T_5 + \mathrm{ILiteral}$.

These results provide a solid logical basis for the study of a wide range of methods of automated
 inductive theorem proving for the language of arithmetic.
Moreover, they also provide several practically relevant independence results, see
 Remark~\ref{remark:t0_prindepresult} and Remark~\ref{remark:t2_prindepresult}.

While for theories without multiplication, quite elementary methods sufficed, for open induction
 in the full language of arithmetic, some ring theory was indispensable.
This is relevant, in particular, for possible extensions of such results to other inductive
 data types such as lists and trees where the lack of a sufficiently developed algebraic
 background theory makes the situation more difficult~\cite{Hetzl24Quantifier,Weiser25Subsystems}.

In order to complete our picture, it would be interesting to study open induction, and its subsystems,
 for extended languages, for example, when an ordering or the divisibility-relation is present.
From the perspective of automated inductive theorem proving, it would be useful to study open induction
 for general inductive data types and for specific examples, such as lists and trees.
Some results in this direction can be found in~\cite{Hetzl24Quantifier,Weiser25Subsystems}.

{\bf Author Contributions Statement.} The authors worked out the results together and wrote the article together.

\bibliographystyle{alpha}
\bibliography{references}

\end{document}